\documentclass[a4paper,12pt,leqno]{article}
\usepackage[left=25mm,top=25mm,right=25mm,bottom=25mm]{geometry}

\usepackage{amsmath}
\usepackage{amsthm}
\usepackage{amssymb}
\usepackage{titlesec}
\usepackage{sectsty}
\usepackage{enumitem}
\usepackage{bm}
\usepackage{mathabx}
\usepackage{hyperref}

\allsectionsfont{\centering\large}

\titlelabel{\thesection.\,\,}

\newtheoremstyle{customdef}
  {\topsep}   
  {\topsep}   
  {\normalfont}  
  {18pt}       
  {\sc} 
  {}          
  {5pt plus 1pt minus 1pt} 
  {\thmname{#1}\thmnumber{ #2}.\thmnote{ (#3)}} 

\theoremstyle{customdef}

\newtheorem*{prob}{Problem}

\newtheoremstyle{customplain}
  {\topsep}   
  {\topsep}   
  {\itshape}  
  {18pt}       
  {\sc} 
  {}          
  {5pt plus 1pt minus 1pt} 
  {\thmname{#1}\thmnumber{ #2}.\thmnote{ (#3)}} 

\theoremstyle{customplain}
\newtheorem{theorem}{Theorem}

\newtheorem{lemma}{Lemma}
\newtheorem{prop}{Proposition}

\begin{document}

\vspace*{0.1cm}

\begin{center}
{\Large \bf \uppercase{On }$\bm{\Lambda^{r}}$\uppercase{-strong convergence of numerical sequences and Fourier series}}
\end{center}

\begin{center}
\uppercase{P.~K\'orus}\\
{\small Department of Mathematics, Juh\'asz Gyula Faculty of Education}\\
{\small  University of Szeged, Hattyas sor 10, H-6725 Szeged, Hungary}\\
{\small e-mail: korpet@jgypk.u-szeged.hu}\\
\end{center}
\bigskip

\begin{abstract} {\small We prove theorems of interest about the recently given $\Lambda^{r}$-strong convergence. We extend the results of F. M\'oricz [On $\Lambda$-strong convergence of numerical sequences and Fourier series, Acta Math.~Hungar., 54 (1989), 319--327].}
\end{abstract}

\let\thefootnote\relax
\let\thefootnote\relax

\section{Introduction}\label{sec1}

Throughout this paper let $\Lambda = \{ \lambda_k: k=0,1,\ldots\}$ be a non-decreasing sequence of positive numbers tending to $\infty$. The concept of $\Lambda$-strong convergence was introduced in \cite{Mor}. We say, that a sequence $S=\{s_k: k=0,1,\ldots\}$ of complex numbers converges $\Lambda$-strongly to a complex number $s$ if 
$$\lim_{n\to \infty} \frac{1}{\lambda_n} \sum_{k=0}^{n} |\lambda_k (s_k - s) - \lambda_{k-1} (s_{k-1} - s) | = 0$$
with the agreement $\lambda_{-1}=s_{-1}=0$. 

It is useful to note that $\Lambda$-strong convergence is an intermediate notion between bounded variation and ordinary convergence.

The following generalization was suggested recently in \cite{Kor}. Throughout this paper, we assume that $r\geq 2$ is an integer. A sequence $S=\{s_k\}$ of complex numbers is said to converge $\Lambda^r$-strongly to a complex number $s$ if 
$$\lim_{n\to \infty} \frac{1}{\lambda_n} \sum_{k=0}^{n} |\lambda_k (s_k - s) - \lambda_{k-r} (s_{k-r} - s)| = 0$$
with the agreement $\lambda_{-1}=\ldots =\lambda_{-r}=s_{-1}=\ldots =s_{-r}=0.$

It was seen that these $\Lambda^r$-convergence notions are intermediate notions between $\Lambda$-strong convergence and ordinary convergence. Two basic results were proved in \cite{Kor}.

\begin{lemma}\label{lemmak1}
A sequence $S$ converges $\Lambda^r$-strongly to a number $s$ if and only if
\begin{itemize}[noitemsep,topsep=0pt]
\item[{\rm (i)}] $S$ converges to $s$ in the ordinary sense, and
\item[{\rm (ii)}] $\displaystyle \lim_{n\to \infty} \frac{1}{\lambda_n} \sum_{k=r}^{n} \lambda_{k-r} |s_k - s_{k-r}| = 0.$
\end{itemize}
\end{lemma}

\begin{lemma}\label{lemmak2}
A sequence $S$ converges $\Lambda^r$-strongly to a number $s$ if and only if
$$\sigma_n := \frac{1}{\lambda_n} \sum_{\substack{0 \leq k \leq n \\ r | n-k}} (\lambda_{k} - \lambda_{k-r}) s_k$$
converges to $s$ in the ordinary sense and condition {\rm (ii)} is satisfied.
\end{lemma}

\section{Results on numerical sequences}\label{sec2}

Denote by $c^r(\Lambda)$ the class of $\Lambda^r$-strong convergent sequences $S=\{s_k\}$ of complex numbers. Obviously, $c^r(\Lambda)$ is a linear space. Let
$$\|S\|_{c^r(\Lambda)}:=\sup_{n \geq 0} \frac{1}{\lambda_n} \sum_{k=0}^{n} |\lambda_k s_k - \lambda_{k-r} s_{k-r}|,$$
and consider the well-known norms
$$\|S\|_{\infty}:=\sup_{k \geq 0} |s_k|, \quad \|S\|_{\rm bv}:=\sum_{k=0}^{\infty} | s_k - s_{k-1}|.$$
It is easy to see that $\|.\|_{c^r(\Lambda)}$ is also a norm on $c^r(\Lambda)$. 

Moreover, one can easily obtain the inequality
$$\sum_{k=0}^{n} |\lambda_k s_k - \lambda_{k-r} s_{k-r}| \leq r \sum_{k=0}^{n} |\lambda_k s_k - \lambda_{k-1} s_{k-1}|$$
and the equality
$$s_k = \frac{1}{\lambda_n} \sum_{\substack{0 \leq k \leq n \\ r | n-k}} (\lambda_k s_k - \lambda_{k-r} s_{k-r}).$$
These together imply the following result.
\begin{prop}\label{prop1}
For every sequence $S=\{s_k\}$ of complex numbers we have
$$\|S\|_{\infty} \leq \|S\|_{c^r(\Lambda)} \leq r \|S\|_{c(\Lambda)} \leq 2r \|S\|_{\rm bv}.$$ As a consequence, ${\rm bv} \subset c(\Lambda) \subset c^r(\Lambda) \subset c$.
\end{prop}

It was seen in \cite{Mor} that $c(\Lambda)$ endowed with the norm $\|.\|_{c(\Lambda)}$ is a Banach space. A similar results holds for $c^r(\Lambda)$.

\begin{theorem}\label{thm1}
The class $c^r(\Lambda)$ endowed with the norm $\|.\|_{c^r(\Lambda)}$ is a Banach space.
\end{theorem}

\begin{proof}[\indent \sc Proof] With an analogous argument to the proof of \cite[Theorem 1]{Mor}, we can get the required completeness of $c^r(\Lambda)$. The only needed modifications are
\begin{gather*}
\frac{1}{\lambda_n} \sum_{k=0}^{n} |\lambda_k (s_{\ell k} - s_k) - \lambda_{k-r} (s_{\ell, k-r} - s_{k-r})|
 \leq \|S_{\ell} - S\|_{\infty} \frac{1}{\lambda_n} \sum_{k=0}^{n} (\lambda_{k} + \lambda_{k-r}) \leq \varepsilon
\end{gather*}
and
\begin{gather*}
\frac{1}{\lambda_n} \sum_{k=0}^{n} |\lambda_k (s_{j k} - s_k) - \lambda_{k-r} (s_{j, k-r} - s_{k-r})|\\
 \leq \frac{1}{\lambda_n} \sum_{k=0}^{n} |\lambda_k (s_{j k} - s_{\ell k}) - \lambda_{k-r} (s_{j, k-r} - s_{\ell, k-r})|\\
 + \frac{1}{\lambda_n} \sum_{k=0}^{n} |\lambda_k (s_{\ell k} - s_k) - \lambda_{k-r} (s_{\ell, k-r} - s_{k-r})| \leq
 \|S_j - S_{\ell}\|_{c^r(\Lambda)} + \varepsilon \leq 2 \varepsilon
\end{gather*}
for large enough $\ell$ and $j$. 
\end{proof}

Now that we saw that $c^r(\Lambda)$ is a Banach space, we show that it has a Schhauder basis. In fact, putting
$$F^{(j)}:=(0,0,\ldots,0,\overset{\undergroup{\scriptstyle j}}{1},0,0,\ldots,0,\overset{\undergroup{\scriptstyle j+r}}{1},0,0,\ldots,0,\overset{\undergroup{\scriptstyle j+2r}}{1},\ldots) \quad (j=0,1,\ldots),$$
clearly each $F^{(j)}\in c^r(\Lambda)$.

\begin{theorem}\label{thm2}
$\{F^{(j)}: j=0,1,\ldots\}$ is a basis in $c^r(\Lambda)$.
\end{theorem}

\begin{proof}[\indent \sc Proof] {\it Existence.} We will show that if $S=\{s_k\}$ is a $\Lambda^r$-strongly convergent sequence, then
\begin{gather}\label{exist}
\lim_{m\to\infty} \|S - \sum_{j=0}^{m} (s_{j} - s_{j-r}) F^{(j)}\|_{c^r(\Lambda)} = 0.
\end{gather}
Since
\begin{gather*}
S - \sum_{j=0}^{m} (s_{j} - s_{j-r}) F^{(j)} \\
= (0,0,\ldots,\overset{\undergroup{\scriptstyle m}}{0},\overset{\undergroup{\scriptstyle m+1}}{s_{m+1} - s_{m-r+1}},{s_{m+2} - s_{m-r+2}},\ldots,\overset{\undergroup{\scriptstyle m+ar+b}}{s_{m+ar+b} - s_{m-r+b}},\ldots),\nonumber
\end{gather*}
where $0 \leq a$, $1\leq b \leq r$, by definition,
{\allowdisplaybreaks
\begin{gather*}
\|S - \sum_{j=0}^{m} (s_{j} - s_{j-r}) F^{(j)}\|_{c^r(\Lambda)} \\ = \sup_{n \geq 1} \frac{1}{\lambda_{m+n}} \Big( \sum_{\substack{1 \leq b \leq r \\ b \leq n}} \lambda_{m+b}|(s_{m+b} - s_{m-r+b}|\\ + \sum_{a=1}^{[n/r]} \sum_{\substack{1 \leq b \leq r \\ ar+b \leq n}} |\lambda_{m+ar+b} (s_{m+ar+b} - s_{m-r+b}) - \lambda_{m+ar-r+b} (s_{m+ar-r+b} - s_{m-r+b})| \Big) \\
\leq \sup_{n \geq 1} \frac{1}{\lambda_{m+n}} \Big( \sum_{a=0}^{[n/r]} \sum_{\substack{1 \leq b \leq r \\ ar+b \leq n}} (\lambda_{m+ar+b} - \lambda_{m+ar-r+b}) |s_{m+ar+b} - s_{m-r+b}|\\
+ \sum_{a=0}^{[n/r]} \sum_{\substack{1 \leq b \leq r \\ ar+b \leq n}} \lambda_{m+ar-r+b} |s_{m+ar+b} - s_{m+ar-r+b}| \Big)\\
\leq r \sup_{j,k > m-r} |s_{j} - s_{k}| + \sup_{n \geq m+1} \frac{1}{\lambda_n} \sum_{k=m+1}^{n} \lambda_{k-r} |s_k - s_{k-r}|.
\end{gather*}
}
Applying Proposition \ref{prop1} and Lemma \ref{lemmak1}, respectively, results in \eqref{exist} to be proved.

{\it Uniqueness.} It can be proved in basically the same way as it was seen in the proof of \cite[Theorem 2]{Mor}.
\end{proof}

\section{Results on Fourier series: $\textit{C}$-metric}\label{sec3}

Denote by $C$ the Banach space of the $2\pi$ periodic complex-valued continuous functions endowed with the norm $\| f \|_C := \max_t |f(t)|$.
Let
\begin{gather}\label{fourC}
\frac{1}{2} a_0 + \sum_{k=1}^{\infty} (a_k (f) \cos{kt} + b_k (f) \sin{kt})
\end{gather}
be the Fourier series of a function $f\in C$ with the usual notation $s_k (f) = s_k (f,t)$ for the $k$th partial sum of the series \eqref{fourC}. Denote by $U$, $A$, and  $S(\Lambda)$, respectively, the classes of functions $f\in C$ whose Fourier series converges uniformly, converges absolutely and converges
uniformly $\Lambda$-strongly on $[0, 2\pi)$, endowed with the usual norms, see \cite{Mor}.

A function $f\in C$ belongs to $S(\Lambda^r)$ if
$$ \lim_{n \to \infty} \left\| \frac{1}{\lambda_n} \sum_{k=0}^{n} |\lambda_k (s_k(f)-f) - \lambda_{k-r} (s_{k-r}(f)-f)| \right\|_C = 0.$$

Set the norm
\begin{align*}
\|f\|_{S(\Lambda^r)}:=\sup_{n \geq 0} \left\| \frac{1}{\lambda_n} \sum_{k=0}^{n} |\lambda_k s_k(f) - \lambda_{k-r} s_{k-r}(f)| \right\|_C,
\end{align*}
which is finite for every for $S(\Lambda^r)$ since
$$\|f\|_{S(\Lambda^r)} \leq \|f\|_C + \sup_{n \geq 0} \left\| \frac{1}{\lambda_n} \sum_{k=0}^{n} |\lambda_k (s_k(f)-f) - \lambda_{k-r} (s_{k-r}(f)-f)| \right\|_C.$$

The norm inequalities corresponding to the ones in Proposition \ref{prop1} are formulated below.

\begin{prop}\label{propC1}
For every function $f\in C$ we have
$$\|f\|_{U} \leq \|f\|_{S(\Lambda^r)} \leq r \|f\|_{S(\Lambda)} \leq 2r \|f\|_{A}.$$ As a consequence, $A \subset S(\Lambda) \subset S(\Lambda^r) \subset U$.
\end{prop}

The following results are the counterparts to Lemmas \ref{lemmak1} and \ref{lemmak2} and Theorems \ref{thm1} and \ref{thm2}, respectively. We omit the details of the analogous proofs, except for Theorem \ref{thmC2}.

\begin{lemma}\label{lemmaC1}
A function $f$ belongs to $S(\Lambda^r)$ if and only if
\begin{itemize}[noitemsep,topsep=0pt]
\item[{\rm (iii)}] $\displaystyle \lim_{k\to \infty} \|s_k (f) - f \|_C = 0$, and
\item[{\rm (iv)}] $\displaystyle \lim_{n\to \infty} \left\|\frac{1}{\lambda_n} \sum_{k=r}^{n} \lambda_{k-r} |s_k (f) - s_{k-r} (f)| \right\|_C = 0.$
\end{itemize}
\end{lemma}

\begin{lemma}\label{lemmaC2}
A function $f$ belongs to $S(\Lambda^r)$ if and only if
\begin{itemize}[noitemsep,topsep=0pt]
\item[{\rm (iii')}] $\displaystyle \lim_{n\to \infty} \|\sigma_n (f) - f \|_C = 0$
\end{itemize}
and condition {\rm (iv)} is satisfied, where 
$$\sigma_n (f) = \sigma_n (f,t) := \frac{1}{\lambda_n} \sum_{\substack{0 \leq k \leq n \\ r | n-k}} (\lambda_{k} - \lambda_{k-r}) s_k (f,t).$$
\end{lemma}

\begin{theorem}\label{thmC1}
The set $S(\Lambda^r)$ endowed with the norm $\|.\|_{S(\Lambda^r)}$ is a Banach space.
\end{theorem}

\begin{theorem}\label{thmC2}
If $f \in S(\Lambda^r)$, then
\begin{align}\label{eqC}
\lim_{m\to \infty} \|s_m (f) - f \|_{S(\Lambda^r)} = 0.
\end{align}
\end{theorem}

\begin{proof}
Since the sequence of partial sums of the Fourier series of the difference $f - s_m (f)$ is
\begin{gather*}
(0,0,\ldots,\overset{\undergroup{\scriptstyle m}}{0},\overset{\undergroup{\scriptstyle m+1}}{s_{m+1} (f) - s_{m} (f)},{s_{m+2} (f) - s_{m} (f)},\ldots),
\end{gather*}
then 
{\allowdisplaybreaks
\begin{gather*}
\|s_m (f) - f \|_{S(\Lambda^r)}\\
= \sup_{n \geq 1} \frac{1}{\lambda_{m+n}} \bigg\| \sum_{\substack{1 \leq b \leq r \\ b \leq n}} \lambda_{m+b}|(s_{m+b}(f) - s_{m}(f)| \\
+ \sum_{a=1}^{[n/r]} \sum_{\substack{1 \leq b \leq r \\ ar+b \leq n}} |\lambda_{m+ar+b} (s_{m+ar+b}(f) - s_{m}(f)) - \lambda_{m+ar-r+b} (s_{m+ar-r+b}(f) - s_{m}(f))| \bigg\|_C \\
\leq \sup_{n \geq 1} \frac{1}{\lambda_{m+n}} \bigg\| \sum_{a=0}^{[n/r]} \sum_{\substack{1 \leq b \leq r \\ ar+b \leq n}} (\lambda_{m+ar+b} - \lambda_{m+ar-r+b}) |s_{m+ar+b}(f) - s_{m}(f)|\\
+ \sum_{a=0}^{[n/r]} \sum_{\substack{1 \leq b \leq r \\ ar+b \leq n}} \lambda_{m+ar-r+b} |s_{m+ar+b}(f) - s_{m+ar-r+b}(f)| \bigg\|_C \\
\leq r \sup_{j,k > m-r} \|s_{j}(f) - s_{k}(f)\|_C + \sup_{n \geq m+1} \bigg\| \frac{1}{\lambda_n} \sum_{k=m+1}^{n} \lambda_{k-r} |s_k (f) - s_{k-r} (f)| \bigg\|_C,
\end{gather*}}
where $0 \leq a$, $1\leq b \leq r$. Applying Proposition \ref{propC1} and Lemma \ref{lemmaC1}, respectively, results in \eqref{eqC} to be proved.
\end{proof}

In the following, our goal is to extend the well-known Denjoy--Luzin theorem presented below (see \cite[p. 232]{Zyg}).

\begin{theorem}[Theorem of Denjoy--Luzin]\label{thmDL}
If 
\begin{gather}\label{sumakbk}
\sum_{k=1}^{\infty} (a_k \cos{kt} + b_k \sin{kt})
\end{gather}
converges absolutely for $t$ belonging to a set $A$ of positive measure, then $\displaystyle \sum_{k=1}^{\infty} (|a_k| + |b_k|)$ converges.
\end{theorem}

This theorem was extended for $\Lambda$-strongly convergent trigonometric series by M\'oricz in \cite{Mor}.

\begin{theorem}\label{thmDLM}
If the $n$th partial sums $s_n (t)$ of the series \eqref{sumakbk} converge $\Lambda$-strongly for $t$ belonging to a set $A$ of positive measure or of second category, then 
\begin{gather}\label{M1}
\lim_{n\to \infty} \frac{1}{\lambda_n} \sum_{k=1}^{n} \lambda_{k-1} (|a_k| + |b_k|) = 0.
\end{gather}
Consequently, if $f\in C$ and the $n$th partial sums $s_n (f,t)$ of the Fourier series \eqref{fourC} converge uniformly $\Lambda$-strongly to $f(t)$ everywhere, then coefficients $a_k = a_k(f)$ and $b_k = b_k(f)$ satisfy \eqref{M1}.
\end{theorem}

First, we extend Theorem \ref{thmDL} for single sine and cosine series.

\begin{theorem}\label{thmDL2}
If 
\begin{gather*}
\sum_{k=1}^{\infty} |a_{2k-1} \cos{(2k-1)t} + a_{2k} \cos{{2k}t}|\quad {\textrm and} \quad \sum_{k=1}^{\infty} |a_{2k-1} \sin{(2k-1)t} + a_{2k} \sin{{2k}t}|
\end{gather*}
converge for $t$ belonging to a set $A$ of positive measure, then $\displaystyle \sum_{k=1}^{\infty} |a_k|$ converges.
\end{theorem}

\begin{proof}
We follow the proof of the Denjoy--Luzin theorem as in \cite[pp. 232]{Zyg} with necessary modifications. We calculate
$$a_{2k-1} \cos{(2k-1)t} + a_{2k} \cos{{2k}t} = (a_{2k-1}+a_{2k}\cos{t}) \cos{(2k-1)t} - (a_{2k}\sin{t}) \sin{(2k-1)t},$$
and
$$a_{2k-1} \sin{(2k-1)t} + a_{2k} \sin{{2k}t} = (a_{2k-1}+a_{2k}\cos{t}) \sin{(2k-1)t} + (a_{2k}\sin{t}) \cos{(2k-1)t},$$
whence
$$a_{2k-1} \cos{(2k-1)t} + a_{2k} \cos{{2k}t} = \rho_k(t) \cos{((2k-1)t+f_k(t))}$$
and
$$a_{2k-1} \sin{(2k-1)t} + a_{2k} \sin{{2k}t} = \rho_k(t) \sin{((2k-1)t+f_k(t))}$$
where 
$$\rho_k(t) = \sqrt{a_{2k-1}^2+a_{2k}^2 + 2a_{2k-1} a_{2k}\cos{t}}$$
and $f_k(t)$ is from
$$\cos{f_k(t)}= \frac{a_{2k-1}+a_{2k}\cos{t}}{\rho_k(t)},\qquad \sin{f_k(t)}= \frac{a_{2k}\sin{t}}{\rho_k(t)}.$$
Now, we need that
\begin{gather}\label{eqb}
\rho_k(t) \geq C (|a_{2k-1}|+|a_{2k}|)
\end{gather}
is satisfied on a set $E \subseteq A$ of positive measure where the constant $C$ is independent of $k$ and $t$. \eqref{eqb} comes from the inequality
$$(1-C^2)(a_{2k-1}^2+a_{2k}^2) \geq 2|a_{2k-1} a_{2k}|(C+|\cos{t}|),$$
hence we just need to define $C$ small enough so that the set $E \subseteq A$ on which
$$|\cos{t}| \leq 1-C-C^2$$
holds is of positive measure. We set $C$ and thereby $E$ that way. Since $E\subseteq A$, there is a set $F \subseteq E$ of positive measure such that
\begin{gather*}
\sum_{k=1}^{\infty} \alpha_k (t) = \sum_{k=1}^{\infty} (|a_{2k-1} \cos{(2k-1)t} + a_{2k} \cos{{2k}t}| + |a_{2k-1} \sin{(2k-1)t} + a_{2k} \sin{{2k}t}|)
\end{gather*}
is bounded on $F$, say by bound $M$. Hence we obtain the required estimation
\begin{gather*}
\sum_{k=1}^{\infty} (|a_{2k-1}|+|a_{2k}|) \leq \frac{1}{C} \sum_{k=1}^{\infty} \int_F \rho_k(t)\\
= \frac{1}{C} \sum_{k=1}^{\infty} \int_F \rho_k(t) (\cos^2{((2k-1)t+f_k(t))} + \sin^2{((2k-1)t+f_k(t))})\, dt\\
\leq \frac{1}{C} \sum_{k=1}^{\infty} \int_F \rho_k(t) (|\cos{((2k-1)t+f_k(t))}| + |\sin{((2k-1)t+f_k(t))}|)\, dt\\ = \frac{1}{C} \sum_{k=1}^{\infty} \int_F \alpha_k (t)\, dt \leq \frac{M}{C} |F|.\qedhere
\end{gather*}
\end{proof}

Second, we extend Theorem \ref{thmDL2} for $\Lambda^{2}$-strong convergent sine or cosine series.

\begin{theorem}\label{thmDL3}
If 
\begin{gather}\label{dl3}
s^1_n(t) = \sum_{k=1}^{n} a_k \cos{kt} \quad {\textrm and} \quad s^2_n(t) = \sum_{k=1}^{n} a_k \sin{kt}
\end{gather}
converge $\Lambda^{2}$-strongly for $t$ belonging to a set $A$ of positive measure, then  
\begin{gather}\label{dl4}
\lim_{n\to \infty} \frac{1}{\lambda_n} \sum_{k=1}^{n} \lambda_{k-1} |a_k| = 0.
\end{gather}
Consequently, if $f,g \in C$ has single Fourier series $\displaystyle \frac{1}{2} a_0 + \sum_{k=1}^{\infty} a_k \cos{kt}$ and $\displaystyle \sum_{k=1}^{\infty} a_k \sin{kt}$, respectively, which partial sums converge uniformly $\Lambda^{2}$-strongly to $f(t)$ and $g(t)$ everywhere, then coefficients $a_k$ satisfy \eqref{dl4}.
\end{theorem}

\begin{proof}
By Lemma \ref{lemmak1} (in the second case Lemma \ref{lemmaC1} is used), $\Lambda^{2}$-strong convergence implies for example, in the cosine case that
$$ \lim_{n\to \infty} \frac{1}{\lambda_n} \sum_{k=2}^{n} \lambda_{k-2} |s^1_k - s^1_{k-2}| = \lim_{n\to \infty} \frac{1}{\lambda_n} \sum_{k=2}^{n} \lambda_{k-2} |a_{k-1} \cos{(k-1)t} + a_{k} \cos{{k}t}| = 0,$$
and the proof is analogous to the one of the previous theorem, Theorem \ref{thmDL2}.
\end{proof}

\section{Results on Fourier series: $\textit{L}^{\textit{p}}$-metric}\label{sec4}

The results of Section \ref{sec3} can be reformulated if we substitute $L^p$-metric for $C$-metric. Here and in the sequel $1 \leq p < \infty$. Along with the usual notations let us call a function $f\in L^p$ to be in $S^p(\Lambda^r)$ if
$$ \lim_{n \to \infty} \left\| \frac{1}{\lambda_n} \sum_{k=0}^{n} |\lambda_k (s_k(f)-f) - \lambda_{k-r} (s_{k-r}(f)-f)| \right\|_p = 0,$$
and introduce the norm
\begin{align*}
\|f\|_{S^p(\Lambda^r)}:=\sup_{n \geq 0} \left\| \frac{1}{\lambda_n} \sum_{k=0}^{n} |\lambda_k s_k(f) - \lambda_{k-r} s_{k-r}(f)| \right\|_p,
\end{align*}
which is finite for every for $S^p(\Lambda^r)$.

The norm inequalities corresponding to the ones in Proposition \ref{propC1} are the following.

\begin{prop}\label{propLp1}
For every function $f\in L^p$ and $r\geq 2$ integer we have
$$\|f\|_{U^p} \leq \|f\|_{S^p(\Lambda^r)} \leq r \|f\|_{S^p(\Lambda)} \leq 2r \|f\|_{A}.$$ As a consequence, $A \subset S^p(\Lambda) \subset S^p(\Lambda^r) \subset U^p$.
\end{prop}

The next results are analogous to Lemmas \ref{lemmaC1} and \ref{lemmaC2} and Theorems \ref{thmC1} and \ref{thmC2}, respectively.

\begin{lemma}\label{lemmaLp1}
A function $f$ belongs to $S^p(\Lambda^r)$ if and only if
\begin{itemize}[noitemsep,topsep=0pt]
\item[{\rm (v)}] $\displaystyle \lim_{k\to \infty} \|s_k (f) - f \|_p = 0$, and
\item[{\rm (vi)}] $\displaystyle \lim_{n\to \infty} \left\|\frac{1}{\lambda_n} \sum_{k=r}^{n} \lambda_{k-r} |s_k (f) - s_{k-r} (f)| \right\|_p = 0.$
\end{itemize}
\end{lemma}

\begin{lemma}\label{lemmaLp2}
A function $f$ belongs to $S^p(\Lambda^r)$ if and only if
\begin{itemize}[noitemsep,topsep=0pt]
\item[{\rm (v')}] $\displaystyle \lim_{n\to \infty} \|\sigma_n (f) - f \|_p = 0$
\end{itemize}
and condition {\rm (vi)} is satisfied.
\end{lemma}

\begin{theorem}\label{thmLp1}
The set $S^p(\Lambda^r)$ endowed with the norm $\|.\|_{S^p(\Lambda^r)}$ is a Banach space.
\end{theorem}

\begin{theorem}\label{thmLp2}
If $f \in S^p(\Lambda^r)$, then
\begin{align*}\label{eqLp}
\lim_{m\to \infty} \|s_m (f) - f \|_{S^p(\Lambda^r)} = 0.
\end{align*}
\end{theorem}

Finally, we obtain the $L^p$-metric version of Theorem \ref{thmDL3}.

\begin{theorem}\label{thmDLp2}
If the sums in \eqref{dl3} converge $\Lambda^{2}$-strongly in the $L^p$-metric restricted to a set of positive measure, then \eqref{dl4} holds true.

Consequently, if $f,g\in L^p$, $1<p<\infty$, has single Fourier series $\displaystyle \frac{1}{2} a_0 + \sum_{k=1}^{\infty} a_k \cos{kt}$ and $\displaystyle \sum_{k=1}^{\infty} a_k \sin{kt}$, respectively, then the partial sums of both series converge $\Lambda^{2}$-strongly to $f(t)$ and $g(t)$ in the $L^p$-metric if and only if coefficients $a_k$ satisfy \eqref{dl4}.
\end{theorem}

\begin{proof}
The first statement and the necessity part of the second statement is obtained in the same way as in the proof of Theorem \ref{thmDL2}. 

The sufficiency part of the second statement follows from two facts. First, by the theorem of Riesz \cite{Zyg}, (v) in Lemma \ref{lemmaLp1} holds. Second, (vi) is also satisfied since
$$
\lim_{n\to \infty} \left\|\frac{1}{\lambda_n} \sum_{k=2}^{n} \lambda_{k-2} |s_k (f) - s_{k-2} (f)| \right\|_p \leq 2 \lim_{n\to \infty} \frac{1}{\lambda_n} \sum_{k=1}^{n} \lambda_{k-1} |a_k| = 0.\eqno\qedhere
$$
\end{proof}

\begin{prob}
Can we prove similar statements to the above proved theorems about the $\Lambda^r$-strong convergence in the case $r>2$ as well? 
\end{prob}


\begin{thebibliography}{1}
\bibitem{Kor} {P.~K\'orus}, {On ${\Lambda^2}$-strong convergence of numerical sequences revisited}, {\it Acta Math.~Hungar.} {\bf 148} (2016), 222--227.
\bibitem{Mor} {F.~M\'oricz}, {On $\Lambda$-strong convergence of numerical sequences and Fourier series}, {\it Acta Math.~Hungar.} {\bf 54} (1989), 319--327.
\bibitem{Zyg} {A.~Zygmund}, {\it Trigonometric series}, Vol.~1, Cambridge Univ. Press, 1959.
\end{thebibliography}
\end{document}